\documentclass[a4paper,reqno]{amsart}

\usepackage{amsmath}
\usepackage{amsthm}
\usepackage{amssymb}
\usepackage{amscd} 

\usepackage{bbm} 
\usepackage{mathtools,mathrsfs} 

\usepackage{verbatim} 

\usepackage{enumitem}

\newtheoremstyle{newremark}
  {5pt}
  {5pt}
  {\rmfamily}
  {}
  {\rmfamily\bf}
  {.}
  {.5em}
  {}

\newtheorem{theorem}{Theorem}
\newtheorem{lemma}[theorem]{Lemma}
\newtheorem{corollary}[theorem]{Corollary}
\newtheorem{proposition}[theorem]{Proposition}

\theoremstyle{newremark}
\newtheorem{remark}[theorem]{Remark}
\newtheorem{definition}[theorem]{Definition}

\newtheorem*{definition*}{Definition} 
\newtheorem*{notations*}{Notations}

\numberwithin{theorem}{section}
\numberwithin{equation}{section}

\newcommand{\R}{\mathbb{R}} 

\def\XXint#1#2#3{{%
\setbox0=\hbox{$#1{#2#3}{\int}$}
\vcenter{\hbox{$#2#3$}}\kern-.5\wd0}}

\renewcommand{\leq}{\leqslant}
\renewcommand{\geq}{\geqslant}
\renewcommand{\subset}{\subseteq}

\newcommand{\eps}{\varepsilon}

\begin{document}


\title[\bf Minimizing fractional harmonic maps on the real line]{Minimizing fractional harmonic maps\\ 
on the real line in the supercritical regime}

\author{Vincent Millot}
\address{Universit\'e Paris Diderot, Lab. J.L.Lions (CNRS UMR 7598), Paris, France}
\email{millot@ljll.univ-paris-diderot.fr}

\author{Yannick Sire}
\address{Johns Hopkins University, Department of Mathematics, Baltimore, USA}
\email{sire@math.jhu.edu}

\author{Hui Yu}
\address{Columbia University, Department of Mathematics, New York, USA}
\email{huiyu@math.columbia.edu}



\begin{abstract}
This article addresses the regularity issue for minimizing fractional harmonic maps of order $s\in(0,1/2)$ from an interval into a smooth manifold. 
H\"older continuity away from a locally finite set is established for a general target. If the target is the standard sphere, then H\"older continuity holds everywhere.  
\end{abstract}



\maketitle




\section{Introduction}

In a series of recent articles  \cite{DLR1,DLR2}, F. Da Lio and T. Rivi\`ere introduced the concept of {\sl $1/2$-harmonic maps} into a manifold. Given a compact smooth submanifold  $\mathcal{N}\subset \mathbb{R}^d$ without boundary, such a  map $u: \mathbb R \to \mathcal N$ is defined as a critical point of the nonlocal energy 
$$
\mathcal E_{\frac{1}{2}}(u):=\frac{1}{2\pi}\iint_{\mathbb R\times\mathbb{R}}\frac{|u(x)-u(y)|^2}{|x-y|^2}\,dxdy\,.
$$
It satisfies the Euler-Lagrange equation
$$(-\Delta)^{\frac{1}{2}}u\perp {\rm Tan}(u,\mathcal{N})\,,$$
where $(-\Delta)^{\frac{1}{2}}$ is the fractional Laplacian as defined in Fourier space. Obviously, this equation is in strong analogy with the standard harmonic map equation into $\mathcal{N}$, and one main issue is to prove {\sl a priori} regularity. This was achieved in  \cite{DLR1,DLR2}, thus extending the famous regularity result of F. H\'elein for classical harmonic maps on surfaces \cite{Hel}. 
The notion of $1/2$-harmonic maps has  been then extended in \cite{MS,Mos} to higher dimensions, and partial regularity for minimizing or stationary $1/2$-harmonic maps established (in analogy with the classical harmonic map problem \cite{Bet,Ev,SU1}). 
\vskip3pt

All these works pave the way to a more general theory for {\sl fractional harmonic maps} where the energy $\mathcal{E}_{\frac{1}{2}}$ is replaced by the Dirichlet form induced by the fractional Laplacian $(-\Delta)^s$ with  exponent $s\in(0,1)$. As noticed in \cite[Remark 1.7]{MSW}, the case $s\in (0,1/2)$ is in strong relation with the so-called {\sl nonlocal minimal surfaces} introduced by L.~Caffarelli, J.M. Roquejoffre, and O. Savin \cite{CRS}. For this reason, we focus here on the case $s\in(0,1/2)$, and as first step toward such a theory, we shall consider minimizing $s$-harmonic maps in one space dimension.  Before going further, let us  give some details on the mathematical framework. 
\vskip3pt

Given $s\in(0,1/2)$ and a bounded open interval $\omega\subset \mathbb{R}$, the nonlocal (or fractional) $s$-energy in $\omega$ of a measurable function $u:\mathbb{R}\to\R^d$ is defined as 
$$\mathcal{E}_s(u,\omega):=\frac{\gamma_s}{2}\iint_{\omega\times\omega}\frac{|u(x)-u(y)|^2}{|x-y|^{1+2s}}\,dxdy +\gamma_s \iint_{\omega\times(\R\setminus\omega)}\frac{|u(x)-u(y)|^2}{|x-y|^{1+2s}}\,dxdy\,.$$
The normalization constant  $\gamma_s:=s2^{2s}\pi^{-\frac{1}{2}}\frac{\Gamma(\frac{1+2s}{2})}{\Gamma(1-s)}$ is chosen in such a way that  
$$\mathcal{E}_s(u,\omega)=\frac{1}{2}\int_{\R}|(-\Delta)^{\frac{s}{2}} u|^2\,dx \quad \forall u\in C^\infty_0(\omega;\R^d)\,.$$
Following \cite[Section 2]{MSW}, we denote by $\widehat H^s(\omega;\R^d)$ 
the Hilbert space made of $L^2_{\rm loc}(\R)$-functions $u$ such that $\mathcal{E}_s(u,\omega)<\infty$, and we set 
$$\widehat H^s(\omega;\mathcal{N}):=\big\{u\in\widehat H^s(\omega;\R^d): u(x)\in \mathcal{N}\text{ a.e. on }\R \big\}\,.$$ 
\begin{definition*}
We say that $u\in \widehat H^s(\omega;\mathcal{N})$ is a {\sl minimizing $s$-harmonic map} in $\omega$ if 
$$\mathcal{E}_s(u,\omega)\leq \mathcal{E}_s(\widetilde u,\omega) $$
for every $\widetilde u\in \widehat H^s(\omega;\mathcal{N})$ such that ${\rm spt}(\widetilde u-u)\subset \omega$. 
\end{definition*}

Exactly as in the case $s=1/2$ (see \cite[Remark 4.24]{MS}), a minimizing $s$-harmonic map satisfies the Euler-Lagrange equation 
$$ (-\Delta)^{s}u\perp {\rm Tan}(u,\mathcal{N})\quad \text{in $\mathscr{D}^\prime(\omega)$}\,.$$
In terms of scaling, this equation turns out to be supercritical (since $s\in(0,1/2)$), and one may expect that minimizing $s$-harmonic maps are singular, exactly as it happens for (classical) minimizing harmonic maps in dimensions greater than three~\cite{SU1}. 
\vskip3pt

The main objective of this paper is  to provide a first partial regularity result for minimizing $s$-harmonic maps. 
At this stage, we should point out that existence is not an issue. Indeed, prescribing an exterior condition $g\in \widehat H^s(\omega;\mathcal{N})$, one can minimize the energy $\mathcal{E}_s(\cdot,\omega)$ over all maps $u\in \widehat H^s(\omega;\mathcal{N})$ satisfying $u=g$ a.e. in $\R\setminus\omega$. Existence for this minimization problem easily follows from the Direct Method of Calculus of Variations, and it obviously produces a minimizing $s$-harmonic map in $\omega$. 
\vskip3pt

Our first main result concerns the case of a general (smooth) target $\mathcal{N}$.

\begin{theorem}\label{main1}
For $s\in (0,1/2)$, let $u\in \widehat H^s(\omega;\mathcal{N})$ be a minimizing $s$-harmonic map in $\omega$. Then $u$ is locally H\"older continuous in $\omega$ away from a locally finite set of points. 
\end{theorem}

The proof of Theorem \ref{main1} follows somehow the general scheme for proving partial regularity of minimizing harmonic maps, or more precisely of minimizing harmonic maps 
with (partially) free boundary. Indeed, the problem can be rephrased as a degenerate regularity problem for harmonic maps 
with free boundary, once we use the so-called Caffarelli-Silvestre extension \cite{CS}.  With this respect, our arguments ressemble to the ones in \cite{DS,HL,HL2}, except that they have 
to be suitably modified to deal with our degenerate setting.  In view of this classical literature, one may wonder if H\"older continuity implies higher order regularity. We do not address  this question here, as it will be the object of a future work. In a complementary direction, one can ask wether or not a (one dimensional) minimizing $s$-harmonic can be singular. We believe that, in general, Theorem \ref{main1} is optimal, but the question remains open. However, if the manifold $\mathcal{N}$ is a standard sphere, then there are no singularities at all. This statement (and proof) is in a sense an amusing fractional counterpart of the regularity result of R. Schoen \& K. Uhlenbeck \cite{SU2} for minimizing harmonic maps into spheres.

\begin{theorem}\label{main2}
For $s\in (0,1/2)$ and $d>1$, let $u\in \widehat H^s(\omega;\mathbb{S}^{d-1})$ be a minimizing $s$-harmonic map in $\omega$. Then $u$ is locally H\"older continuous in $\omega$. 
\end{theorem}

This article is organized as follows. In Section \ref{weightharmap}, we introduce the notion of harmonic maps with free boundary induced by the Caffarelli-Silvestre extension, together with some fundamental properties  such as the monotonicity formula. In Section \ref{epsregsec}, we prove an $\varepsilon$-regularity theorem for those harmonic maps with free boundary. Section~\ref{compactnes} is devoted to compactness properties of minimizing $s$-harmonic maps, and Theorems~\ref{main1} \& \ref{main2} are proved in Section \ref{prfthms}.  

\subsection*{Notation} We shall often identify $\R$ with $\partial  \mathbb{R}^{2}_+=\R\times\{0\}$. More generally, a set $A\subset\mathbb{R}$ can be identified with $A\times\{0\}\subset\partial  \mathbb{R}^{2}_+$. Points in $\mathbb{R}^{2}$ are written $\mathbf{x}=(x,y)$. 
We denote by $B_r(\mathbf{x})$ the open disc in $\mathbb{R}^{2}$ of radius $r$ centered at $\mathbf{x}=(x,y)$.  
For an arbitrary set $\Omega\subset  \mathbb{R}^{2}$, we write $\Omega^+:=\Omega\cap \mathbb{R}^{2}_+$ and $\partial^+ \Omega:=\partial \Omega\cap \mathbb{R}^{2}_+$. 

If $\Omega\subset\R^{2}_+$ is a bounded open set, we shall say that $\Omega$ is  {\bf admissible} whenever 
\begin{itemize}
\item $\partial \Omega$ is Lipschitz regular;  
\vskip2pt
\item the (relative) open set $\partial^0\Omega\subset \partial\R^2_+\simeq\R $ is defined by 
$$\partial^0\Omega:=\Big\{\mathbf{x}\in\partial \Omega\cap\partial\R^{2}_+ : B^+_{r}(\mathbf{x})\subset \Omega \text{ for some $r>0$}\Big \}\,,$$
is non empty and has Lipschitz boundary; 
\vskip2pt

\item $\partial \Omega=\partial^+ \Omega\cup\overline{\partial^0\Omega}\,$.
\end{itemize}
\vskip3pt

\noindent Finally, we denote by ${d}_{\mathcal N}$ the distance function on $\R^d$ to the manifold $\mathcal N$, i.e., 
$${d}_{\mathcal N}(z):=\inf_{p\in \mathcal{N}} |z-p|\,.$$
The tangent and normal spaces to $\mathcal{N}$ at a point $p\in\mathcal{N}$ are denoted by ${\rm Tan}(p,\mathcal{N})$ and ${\rm Nor}(p,\mathcal{N})$, respectively.

\section{Minimizing weighted harmonic maps with free boundary} \label{weightharmap}     

The proof of our results relies on  the already mentioned Caffarelli-Silvestre extension procedure \cite{CS} which allows to rephrase our fractional problem into a local one. Before going into details on the extension of minimizing $s$-harmonic maps, we briefly introduce the resulting local problem and its functional setting.

\subsection{Minimizing weighted harmonic maps} For a bounded admissible open subset $\Omega\subset \R^2_+$, we consider the weighted spaces
$$L^2(\Omega;\R^d,y^ad{\bf x}):=\Big\{v\in L^1_{\rm loc}(\Omega;\R^d): y^{\frac{a}{2}}|v|\in L^2(\Omega)\Big\}\quad \text{with } a:=1-2s>0\,, $$
and 
$$H^1(\Omega;\R^d,y^ad{\bf x}):=\Big\{v\in L^2(\Omega,y^ad{\bf x}): \nabla v\in L^2(\Omega,y^ad{\bf x})\Big\}\,.$$ 
We refer to \cite[Section 2]{MSW} for the main properties of these spaces that we shall use. We simply recall that 
a map $v\in H^1(\Omega;\R^d,y^ad{\bf x})$ has a well defined trace on $\partial^0\Omega$, and the trace operator from $ H^1(\Omega;\R^d,y^ad{\bf x})$ into $L^2(\partial^0\Omega;\R^d)$ is a compact linear operator. 
\vskip3pt
 
 On $H^1(\Omega,y^adxdy)$, we define the weighted Dirichlet energy
 $${\bf E}_s(v,\Omega):=\frac{1}{2}\int_\Omega y^a|\nabla v|^2\,d{\bf x}\,. $$

\begin{definition}
Let $\Omega\subset \R^2_+$ be a bounded admissible open set, and consider a map $v\in H^1(\Omega;\R^d,y^ad{\bf x})$ such that $v({\bf x})\in \mathcal{N}$  a.e. on $\partial^0 \Omega$.  We say that $v$ is a minimizing weighted harmonic map in $\Omega$ with respect to the (partially) free boundary  $v(\partial^0\Omega)\subset \mathcal{N}$ if$$ {\bf E}_s(v,\Omega)\leq {\bf E}_s(w,\Omega)$$
for every competitor $w\in H^1(\Omega,y^ad{\bf x})$ satisfying $w({\bf x})\in \mathcal{N}$ a.e. on $\partial^0 \Omega$, and such that ${\rm spt}(w-v)\subset \Omega\cup\partial^0\Omega$. 
{\sl In short, we shall say $u$ is a minimizing weighted harmonic map with free boundary in $\Omega$.} 
\end{definition}

\begin{remark}
Using variations supported in $\Omega$, one obtains that a minimizing weighted harmonic map $v$ with free boundary in $\Omega$ satisfies
\begin{equation}\label{yharmeq}
{\rm div}(y^a\nabla v)=0\quad\text{in $\Omega$}\,.
\end{equation}
In particular, $v\in C^\infty(\Omega)$ by standard elliptic theory. The regularity issue is  then at (and only at) the free boundary $\partial^0\Omega$. Arguing exactly as \cite[Section 2]{DS}, one obtains 
$$y^a\partial_y v\perp {\rm Tan}(v,\mathcal{N})\quad\text{on $\partial^0\Omega$}$$ 
in the duality sense. In other words, the (full) Euler-Lagrange equation derived from minimality is 
\begin{equation}\label{ELfreebound}
\int_\Omega y^a\nabla v\cdot \nabla \zeta\,d{\bf x}=0
\end{equation}
for every $\zeta\in H^1(\Omega;\R^d,y^ad{\bf x})$  satisfying $\zeta({\bf x})\in {\rm Tan}(v({\bf x}),\mathcal{N})$ for a.e. ${\bf x}\in\partial^0\Omega$, and such that ${\rm spt}(\zeta)\subset\Omega\cup\partial^0\Omega$. 
\end{remark}

\subsection{Extending minimizing $s$-harmonic maps}

We now move on the extension procedure of \cite{CS}. Given a bounded open interval $\omega\subset \R$, we define the extension $u^{\rm e}:\R^2_+\to\R^d$ of a map $u\in \widehat H^s(\omega;\R^d)$ by 
$$u^{\rm e}(x,y):=\sigma_{s}\int_\R \frac{t^{2s} u(t)}{(|x-y|^2+t^{2})^{\frac{1+2s}{2}}}\,dt\quad \text{with }\sigma_s:=\pi^{-\frac{1}{2}}\frac{\Gamma(\frac{1+2s}{2})}{\Gamma(s)}\,.$$
This extension can be referred to as {\sl fractional harmonic extension} of $u$ (by analogy with the case $s=1/2$) as it solves
\begin{equation}\label{eqextint}
\begin{cases}
{\rm div}(y^a\nabla u^{\rm e})=0 & \text{in $\R^2_+$}\,,\\
u^{\rm  e}=u & \text{on $\R\simeq\partial \R^2_+$}\,.
\end{cases}
\end{equation}
It turns out that $u^{\rm e}\in H^1(\Omega;\R^d,y^ad{\bf x})$ for every bounded admissible open set $\Omega\subset \R^2_+$ such that $\overline{\partial^0\Omega}\subset \omega$. In addition, $u^{\rm e}\in L^\infty(\R^2_+)$ whenever $u\in L^\infty(\R)$, and $\|u^{\rm e}\|_{L^\infty(\R^2_+)}\leq \|u\|_{L^\infty}(\R)$. We refer to \cite[Section 2]{MSW} for further details. 
\vskip3pt

We shall make use of the following converse statement to control the fractional energy by the weighted Dirichlet energy. 

\begin{lemma}\label{H1tofracesti}
Let $r>0$ and $v\in H^1(B^+_{2r};\R^d, y^ad{\bf x})$. The trace of $v$ on the interval $\omega_r:=\partial^0B^+_r$ belongs to $H^s(\omega_r;\R^d)$, and 
$$\iint_{\omega_r\times\omega_r}\frac{|v(x)-v(y)|^2}{|x-y|^{1+2s}}\,dxdy\leq  C{\bf E}_s(v,B^+_{2r})\,.$$
for some constant $C=C(s)$.
\end{lemma}

\begin{proof}
Without loss of generality, we may assume that $v$ has a vanishing average over the half ball $B^+_{2r}$. Let $\zeta\in C^\infty(B_{2r};[0,1])$ be a cut-off function such that $\zeta({\bf x})=1$ for $|{\bf x}|\leq r$,  $\zeta({\bf x})=0$ for $|{\bf x}|\geq 3r/2$, and satisfying $|\nabla\zeta|\leq C/r$. The function $v_r:=\zeta v$ belongs to $H^1(\R^2_+;\R^d, y^ad{\bf x})$, and Poincar\'e's inequality in $H^1(\R^2_+;\R^d, y^ad{\bf x})$ (see e.g.~\cite{FKS}) yields 
\begin{equation}\label{esticpa1}
\int_{\R^2_+}y^a|\nabla v_r|^2\,d{\bf x}\leq 2{\bf E}_s(v,B^+_{2r})+ \frac{C}{r^2} \int_{B^+_{2r}}y^a|v|^2\,d{\bf x}\leq C_s{\bf E}_s(v,B^+_{2r})\,.
\end{equation}
On the other hand, it follows from \cite{CS} (see also \cite[Lemma 2.8]{MSW}) that 
\begin{multline}\label{esticpa2}
\iint_{\omega_r\times\omega_r}\frac{|v(x)-v(y)|^2}{|x-y|^{1+2s}}\,dxdy\leq \iint_{\R\times\R}\frac{|v_r(x)-v_r(y)|^2}{|x-y|^{1+2s}}\,dxdy\\
\leq C_s \int_{\R^2_+}y^a|\nabla v_r|^2\,d{\bf x}\,,
\end{multline}
for some constant $C_s$ depending only on $s$. Gathering \eqref{esticpa1} and \eqref{esticpa2} leads to the announced estimate.
\end{proof}

The following proposition  draws links between minimizing $s$-harmonic maps and minimizing weighted harmonic maps with free boundary.  
Its proof follows exactly as in \cite[Proposition 4.9]{MS} (see also \cite[Corollary 2.13]{MSW}), and we shall omit it. 

\begin{proposition}
Let $\omega\subset\R$ be a bounded open interval, and $u\in \widehat H^s(\omega;\mathcal{N})$  a minimizing $s$-harmonic map in $\omega$. Then $u^{\rm e}$ is a minimizing weighted harmonic map in $\Omega$ with free boundary  in every bounded admissible open set $\Omega\subset\R^2_+$ satisfying $\overline{\partial^0\Omega}\subset \omega$. 
\end{proposition}

\subsection{The monotonicity formula}

In this subsection, we consider a bounded admissible open set $\Omega\subset\R^2_+$, and a minimizing weighted harmonic map $v\in H^1(\Omega;\R^d,y^ad{\bf x})$ with free boundary. We present the fundamental monotonicity formula involving the following density function: for a point ${\bf x}_0=(x_0,0)\in\partial^0\Omega$ and $r>0$ such that $B^+_r({\bf x_0})\subset \Omega$ , we set 
$$\mathbf{\Theta}_v({\bf x}_0,r):=\frac{1}{r^{1-2s}}{\bf E}_s\big(v,B^+_r({\bf x}_0)\big)\,. $$ 

\begin{lemma}\label{formmonot}
For every ${\bf x}_0\in\partial^0\Omega$ and $r>\rho>0$ such that $B^+_r({\bf x}_0)\subset \Omega$, 
$$ \mathbf{\Theta}_v({\bf x}_0,r)-\mathbf{\Theta}_v({\bf x}_0,\rho)=\int_{B^+_r({\bf x}_0)\setminus B^+_\rho({\bf x}_0)}y^a\frac{|({\bf x}-{\bf x}_0)\cdot\nabla v|^2}{|{\bf x}-{\bf x}_0|^{3-2s}}\,d{\bf x}\,.$$
\end{lemma}

\begin{proof}
The proof follows classically from the stationarity implied by minimality. To be more precise, let us consider a vector field ${\bf X}=({\bf X_1},{\bf X}_2)\in C^1(\overline{\R^2}_+;\R^2)$ compactly supported in $\Omega\cup\partial^0\Omega$ and such that ${\bf X}_2=0$ on $\R\times\{0\}$. Then consider a compactly supported $C^1$-extension of ${\bf X}$ to the whole $\R^2$, still denoted by ${\bf X}$. We define $\{\phi_t\}_{t\in\R}$ the flow on $\R^2$ generated by ${\bf X}$, i.e., for ${\bf x}\in\R^2$, the map $t\mapsto \phi_t({\bf x})$ is defined as the unique solution of the differential equation
$$\begin{cases}
\displaystyle \frac{d}{dt} \phi_t({\bf x})= {\bf X}\big(\phi_t(x)\big)\,,\\
\phi_0({\bf x})={\bf x}\,.
\end{cases}$$
Notice that $\phi_t(\Omega)=\Omega$, $\phi_t(\partial^0\Omega)=\partial^0\Omega$, and ${\rm spt}(\phi_t-{\rm id}_{\R^2})\cap \overline{\R^2}_+\subset \Omega\cup\partial^0\Omega$. As a consequence, the map $v_t:=v\circ\phi_t\in H^1(\Omega;\R^d, y^a\,d{\bf x})$ satisfies $v_t({\bf x})\in \mathcal{N}$ a.e. on $\partial^0\Omega$, and ${\rm spt}(v_t-v)\subset\Omega\cup\partial^0\Omega$. 
By minimality ${\bf E}_s(v,\Omega)\leq {\bf E}_s(v_t)$ for every $t\in\R$, so that 
$$\left[\frac{d}{dt}{\bf E}_s(v_t,\Omega)\right]_{t=0}=0\,. $$
Computing this derivative (see e.g. \cite[Chapter 2.2]{Sim} or \cite{MSW}) leads to 
\begin{equation}\label{fromstat}
\int_\Omega y^a \left(|\nabla v|^2{\rm div}\,{\bf X}-2\sum_{i,j=1}^2(\partial_iv\cdot\partial_j v)\partial_j{\bf X}_i\right)\,d{\bf x}+a\int_\Omega y^{a-1}|\nabla v|^2{\bf X}_2\,d{\bf x}=0
\end{equation}
for every vector field ${\bf X}=({\bf X_1},{\bf X}_2)\in C^1(\overline{\R^2}_+;\R^2)$ compactly supported in $\Omega\cup\partial^0\Omega$ and such that ${\bf X}_2=0$ on $\R\times\{0\}$. From equation \eqref{fromstat}, the announced monotonicity  follows as in \cite[Lemma 6.2]{MSW}. 
\end{proof}

\begin{corollary}\label{corodens}
For every ${\bf x}_0\in\partial^0\Omega$, the limit
$$\mathbf{\Theta}_v({\bf x}_0):=\lim_{r\downarrow 0}  \mathbf{\Theta}_v({\bf x}_0,r)$$
exists, and the function $\mathbf{\Theta}_v:\partial^0\Omega\to [0,\infty)$ is upper semicontinuous. In addition, 
\begin{equation}\label{sam2223}
 \mathbf{\Theta}_v({\bf x}_0,r)-\mathbf{\Theta}_v({\bf x}_0)=\int_{B^+_r({\bf x}_0)}y^a\frac{|({\bf x}-{\bf x}_0)\cdot\nabla v|^2}{|{\bf x}-{\bf x}_0|^{3-2s}}\,d{\bf x}\,.
 \end{equation}
\end{corollary}

\begin{proof}
The existence of the limit defining $\mathbf{\Theta}_v$ as well as \eqref{sam2223}  are straightforward consequences of Lemma \ref{formmonot}. Then $\mathbf{\Theta}_v$ is upper semicontinuous as a pointwise limit of a decreasing family of continuous functions. 
\end{proof}

\section{The $\eps$-regularity theorem} \label{epsregsec}  

\subsection{An extension lemma and the hybrid inequality} This subsection is essentially devoted to the construction of comparison maps. We shall start with the construction of competitors from a boundary data satisfying a small oscillation condition. Testing minimality against such competitors leads to the so-called {\sl hybrid inequality} (see \cite{HL,HL2}), a central estimate in the proof of the $\varepsilon$-regularity theorem. 
\vskip3pt

Let us start with an elementary lemma. 

\begin{lemma}\label{firstesti}
Let $v\in H^1(\partial^+ B_1;\mathbb{R}^{d},y^ad{\bf x})$ be such that  $v(\pm1,0)\in \mathcal{N}$. Then, 
$$d^2_{\mathcal N}\big(v({\bf x})\big)\leq \left(\int_{\partial^+ B_1}y^{a}|\partial_\tau v|^2\,d\mathcal{H}^1\right)^{1/2}\left(\int_{\partial^+ B_1}y^{-a}d^2_{\mathcal N}(v)\,d\mathcal{H}^1\right)^{1/2}$$ 
for every ${\bf x}\in\partial^+ B_1$.
\end{lemma} 

\begin{proof}
First recall that $H^1(\partial^+ B_1;\mathbb{R}^{d},y^ad{\bf x})\hookrightarrow W^{1,1}(\partial^+ B_1;\mathbb{R}^{d})$, so that maps in $H^1(\partial^+ B_1;\mathbb{R}^{d},y^ad{\bf x})$ are continuous on $\overline{\partial^+B_1}$. Then notice that the function $d_{\mathcal{N}}$ is $1$-Lipschitz, and by chain rule one derives $|\nabla d^2_{\mathcal{N}}|\leq 2 d_{\mathcal{N}}$ a.e. in $\R^d$. In turns, it implies that 
$d^2_{\mathcal{N}}(u) \in H^1(\partial^+ B_1,y^ad{\bf x})$ and 
$$|\partial_\tau d^2_{\mathcal{N}}(v)|\leq 2 d_{\mathcal{N}}(v)|\partial_\tau v| \quad \text{a.e. on $\partial^+B_1$}\,.$$
Since $v(-1,0)\in\mathcal{N}$, this estimate implies that for every ${\bf x}\in\partial^+ B_1$, 
$$d^2_{\mathcal N}\big(v({\bf x})\big)\leq 2\int_{((-1,0),{\bf x})}d_{\mathcal{N}}(v)|\partial_\tau v|\,d\mathcal{H}^1 \,,$$
where $((-1,0),{\bf x})$ denotes the arc in $\partial^+ B_1$ going from $(-1,0)$ to ${\bf x}$. The announced inequality  then follows from Cauchy-Schwarz inequality. 
\end{proof}

\begin{lemma}[\bf Comparison maps]\label{lemext}
There exist two constants $\varepsilon_0=\varepsilon_0(\mathcal{N})>0$ and $C=C(\mathcal{N},s)$ such that the following holds. 
Let $v\in H^1(\partial^+ B_1;\mathbb{R}^{d},y^ad{\bf x})$ be such that $v(\pm 1,0)\in \mathcal{N}$, and $\xi\in \R^d$. 
If 
\begin{equation}\label{smallcondext}
\left(\int_{\partial^+ B_1}y^{a}|\partial_\tau v|^2\,d\mathcal{H}^1\right)\left(\int_{\partial^+ B_1}y^{-a}|v-\xi|^2\,d\mathcal{H}^1+\int_{\partial^+ B_1}y^{-a}d^2_{\mathcal N}(v)\,d\mathcal{H}^1\right)\leq\varepsilon_0\,,
\end{equation}
then there exists $w\in H^1(B_1^+;\mathbb{R}^{d}, ,y^ad{\bf x})$ such that $w(\partial^0B^+_1)\subset \mathcal N$, $w=v$  on $\partial^+ B_1$, and 
\begin{multline*}
\int_{B_1^+}y^{a}|\nabla w|^2\,d{\bf x}\\\leq C\left (\int_{\partial^+ B_1}y^{a}|\partial_\tau v|^2\,d\mathcal{H}^1\right )^{1/2}\left (\int_{\partial^+ B_1}y^{-a}|v-\xi|^2\,d\mathcal{H}^1+
\int_{\partial^+ B_1}y^{-a}d^2_{\mathcal N}(v)\,d\mathcal{H}^1\right)^{1/2}\,.
\end{multline*}
\end{lemma} 

\begin{proof}
Reflect $v$ evenly to the entire sphere $\partial B_1$. Obviously, $v\in H^1(\partial B_1;\mathbb{R}^{d},|y|^ad{\bf x})$ by symmetry.  
We consider the variational solution $h\in H^1(B_1;\mathbb{R}^{d},|y|^ad{\bf x})$ of 
 \begin{equation}\label{yharmeqB}
 \begin{cases} 
{\rm  div}(|y|^{a}\nabla h)=0 &\text{in $B_1$}\,,\\
h=v &\text{on $\partial B_1$}\,. 
 \end{cases}
 \end{equation}
 Note that $h\in L^\infty(B_1)$. Indeed, since $v$ is absolutely continuous, it is bounded. Since $h$ minimizes ${\bf E}_s(\cdot,B_1)$ over all maps equal to $v$ on $\partial B_1$, 
 a classical truncation argument shows that $|h|$ does not exceed $\|v\|_{L^\infty}(\partial B_1)$. 
 \vskip3pt
 
Next, recalling \cite[Lemma 6.2]{CS}, we have 
 \begin{equation}\label{estL2normderiv}
 \int_{\partial B_1}|y|^{a}|\partial_\nu h|^2\,d\mathcal{H}^1\leq \int_{\partial B_1}|y|^{a}|\partial_\tau v|^2\,d\mathcal{H}^1\,.
 \end{equation}
Using  \eqref{yharmeqB} and \eqref{estL2normderiv}, we infer from the divergence theorem and Cauchy-Schwarz inequality  that 
\begin{align*}
\int_{B_1}|y|^{a}|\nabla h|^2\,d{\bf x}&=\int_{B_1}|y|^{a}|\nabla (h-\xi)|^2\,d{\bf x}\\
&=\int_{\partial B_1}|y|^{a}(h-\xi)\cdot \partial_\nu h\,d\mathcal{H}^1\\
&\leq \left(\int_{\partial B_1}|y|^{a}|v-\xi|^2\,d\mathcal{H}^1\right)^{1/2}\left(\int_{\partial B_1}|y|^{a}|\partial_\tau v|^2\,d\mathcal{H}^1\right)^{1/2}\,.
\end{align*}
Hence, by symmetry, 
\begin{equation}\label{contrenergh}
\int_{B_1}|y|^{a}|\nabla h|^2\,d{\bf x}\leq 2\left(\int_{\partial^+ B_1}y^{a}|v-\xi|^2\,d\mathcal{H}^1\right)^{1/2}\left(\int_{\partial^+ B_1}y^{a}|\partial_\tau v|^2\,d\mathcal{H}^1\right)^{1/2}\,.
\end{equation}
By the fundamental theorem of calculus (and symmetry), we have   
\begin{align*}
|v({\bf x})-v({\bf x}_0)|^2&\leq 2 \int_{({\bf x}_0,{\bf x}_1)}|v-v({\bf x}_0)||\partial_\tau v|\,d\mathcal{H}^1\\
&\leq 2\left(\int_{\partial^+ B_1}y^{-a}|v-v({\bf x}_0)|^2\,d\mathcal{H}^1\right)^{1/2}\left(\int_{\partial^+ B_1}y^{a}|\partial_\tau v|^2\,d\mathcal{H}^1\right)^{1/2}
\end{align*}
for every ${\bf x},{\bf x}_0\in\partial B_1$. 

We choose the point ${\bf x}_0$ in such a way that ${\bf x}\mapsto |v({\bf x})-\xi|$ achieves its minimum at ${\bf x}_0$. Then,  
$$|v-v({\bf x}_0)|^2\leq 2|v-\xi|^2+2|v({\bf x_0})-\xi|^2\leq 4|v-\xi|^2\quad\text{on $\partial B_1$}\,.$$
Consequently, 
\begin{equation}\label{contrubd}
|v({\bf x})-v({\bf x}_0)|^2\leq 4\left(\int_{\partial^+ B_1}y^{-a}|v-\xi|^2\,d\mathcal{H}^1\right)^{1/2}\left(\int_{\partial^+ B_1}y^{a}|\partial_\tau v|^2\,d\mathcal{H}^1\right)^{1/2}
\end{equation}
for every ${\bf x}\in\partial B_1$.

Since $h$ is bounded, $|h-v({\bf x}_0)|^2$ belongs to $H^1(B_1,|y|^ad{\bf x})$. Moreover, 
$${\rm div}\big(|y|^a\nabla(|h-v({\bf x}_0)|^2)\big)\geq 0 \quad\text{in $B_1$}\,,$$
and  the maximum principle in \cite{FKS}  together with \eqref{contrubd} implies that 
$$|h({\bf x})-v({\bf x}_0)|^2\leq  4\left(\int_{\partial^+ B_1}y^{-a}|v-\xi|^2\,d\mathcal{H}^1\right)^{1/2}\left(\int_{\partial^+ B_1}y^{a}|\partial_\tau v|^2\,d\mathcal{H}^1\right)^{1/2}$$
for every ${\bf x}\in B_1$. Applying Lemma \ref{firstesti} at ${\bf x}={\bf x}_0$, we now infer that 
\begin{align*}
d_{\mathcal N}\big(h({\bf x})\big) &\leq \big|d_{\mathcal N}\big(h({\bf x})\big)-d_{\mathcal N}\big(v({\bf x}_0)\big)\big|+d_{\mathcal N}\big(v({\bf x}_0)\big)\\
&\leq  \big|h({\bf x})-v({\bf x}_0)\big|+d_{\mathcal N}\big(v({\bf x}_0)\big)\\
&\leq 2\left(\int_{\partial^+ B_1}y^{-a}|v-\xi|^2\,d\mathcal{H}^1\right)^{1/4}\left(\int_{\partial^+ B_1}y^{a}|\partial_\tau v|^2\,d\mathcal{H}^1\right)^{1/4}\\
&\quad+\left(\int_{\partial^+ B_1}y^{-a}d^2_{\mathcal N}(v)\,d\mathcal{H}^1\right)^{1/4}\left(\int_{\partial B_1^+}y^{a}|\partial_\tau v|^2\,d\mathcal{H}^1\right)^{1/4}\\
&\leq 2\left(\int_{\partial^+ B_1}y^{a}|\partial_\tau v|^2\,d\mathcal{H}^1\right)^{1/4}\bigg(\int_{\partial^+ B_1}y^{-a}|v-\xi|^2\,d\mathcal{H}^1\\
&\hskip150pt+\int_{\partial^+ B_1}y^{-a}d^2_{\mathcal N}(v)\,d\mathcal{H}^1\bigg)^{1/4}
\end{align*}
for every ${\bf x}\in B_1$. By our assumption, we thus have
$$d_{\mathcal N}(h)\leq 2 \varepsilon_0^{1/4}\quad\text{in $B_1$}\,.$$
As a consequence, if $\varepsilon_0=\varepsilon_0(\mathcal{N})$ is small enough,  $h$ takes values in a small tubular neighborhood of $\mathcal{N}$. In such a neighborhood, the nearest point retraction $\pi_{\mathcal N}$ on $\mathcal N$ is well defined and smooth. Therefore, $\pi_{\mathcal N}(h)$ belongs to $H^1(B_1;\mathcal{N}, |y|^a\,d{\bf x})$, and 
\begin{multline}\label{pinhmoinsh}
\|\pi_{\mathcal{N}}(h)-h\|_{L^\infty(B_1)}^2\leq 4\left(\int_{\partial^+ B_1}y^{a}|\partial_\tau v|^2\,d\mathcal{H}^1\right)^{1/2}\bigg(\int_{\partial^+ B_1}y^{-a}|v-\xi|^2\,d\mathcal{H}^1\\
+\int_{\partial^+ B_1}y^{-a}d^2_{\mathcal N}(v)\,d\mathcal{H}^1\bigg)^{1/2}\,.
\end{multline}
We shall now construct the extension $w$ interpolating $h$ and $\pi_{\mathcal{N}}(h)$ near $\partial^+B_1$. We proceed as follows.  Consider the set 
$$A:=\Big\{{\bf x}=(x,y)\in\overline B_1^+: 0\leq y\leq 1/2\,,\;|x|\leq \sqrt{1-y^2} - y^{\frac{a}{2}}\Big\}\,, $$
and let $\zeta\in C^\infty(B_1^+;[0,1])$ be a smooth cut-off function satisfying $\zeta=1$ in $A\cap B_1^+$, and $\zeta=0$ on $\partial^+B_1$.  From the very definition of $A$, we can even find $\zeta$ in such a way that 
\begin{equation}\label{contrcutoff}
|\partial_y\zeta(x,y)|\leq C \quad\text{and}\quad |\partial_x\zeta(x,y)|\leq Cy^{\frac{-a}{2}} \,,
\end{equation}
where $C=C(s)$. In particular, $\zeta\in H^1(B_1^+;[0,1], y^ad{\bf x})$.  We finally define 
$$w:=\zeta\big(\pi_{\mathcal{N}}(h)-h\big)+h\in H^1(B_1^+;\R^d,y^ad{\bf x})\,.$$
By construction, $w({\bf x})\in\mathcal{N}$ for ${\bf x}\in\partial^0B_1^+$, and $w=h=v$ on $\partial^+B_1$. Then we estimate
\begin{align}
\nonumber \int_{B^+_1}y^{a}|\nabla w|^2\,d{\bf x}& \leq 2\int_{B^+_1}y^{a}|\nabla h|^2\,d{\bf x}+4\int_{B^+_1}y^{a}\big|\nabla(\pi_{\mathcal{N}}(h)-h)\big|^2\,d{\bf x}\\
\nonumber&\hskip120pt +4\int_{B^+_1}y^{a}|\nabla \zeta|^2|\pi_\mathcal{N}(h)-h|^2\,d{\bf x}\\
\label{finestext} & \leq C\int_{B^+_1}y^{a}|\nabla h|^2\,d{\bf x}+4\|\pi_\mathcal{N}(h)-h\|_{L^\infty(B_1^+)}^2\int_{B^+_1}y^{a}|\nabla \zeta|^2\,d{\bf x}\,.
 \end{align}
 Gathering \eqref{finestext} with \eqref{contrenergh}, \eqref{pinhmoinsh}, and \eqref{contrcutoff} leads to the announced result. 
\end{proof}

\begin{corollary}[\bf Hybrid inequality]\label{hybrid}
There exists two constants $\varepsilon_1=\varepsilon_1(\mathcal{N})>0$ and $C=C(\mathcal{N},s)$ such that the following holds. 
Let $v\in H^1(B_1^+;\R^d,y^ad{\bf x})$ be a minimizing weighted harmonic map with free boundary in $B_1^+$, and $\xi\in \R^d$. If 
$$\left(\int_{B_1^+}y^{a}|\nabla v|^2\,d{\bf x}\right)\left(\int_{B_1^+}y^{-a}|v-\xi|^2\,d{\bf x}+\int_{B_1^+}y^{-a}d^2_{\mathcal N}(v)\,d{\bf x}\right)\leq \varepsilon_1\,,$$ 
then
\begin{multline*}
\int_{B_{1/2}^+}y^{a}|\nabla v|^2\,d{\bf x}\leq \lambda\int_{B_1^+}y^{a}|\nabla v|^2\,d{\bf x}\\
+\frac{C}{\lambda}\left(\int_{B_1^+}y^{-a}|v-\xi|^2\,d{\bf x}
+\int_{B_1^+}y^{-a}d^2_{\mathcal N}(v)\,d{\bf x}\right)
\end{multline*}
for every $\lambda\in(0,1)$. 
\end{corollary}

\begin{proof}
By a classical averaging argument, we can find $r\in(1/2,1)$ such that $v$ restricted to $\partial^+B_r$ belongs to $H^1(\partial^+B_r;\R^d,y^a\,d{\bf x})$, and 
\begin{align*}
\int_{\partial^+ B_r}y^{-a}|v-\xi|^2\,d\mathcal{H}^1& \leq 12\int_{B_1^+}y^{-a}|u-\xi|^2\,d{\bf x}\,,\\ 
\int_{\partial^+ B_r}y^{-a}d^2_{\mathcal N}(v)\,d\mathcal{H}^1& \leq 12\int_{B_1^+}y^{-a}d^2_{\mathcal N}(v)\,d{\bf x}\,,\\
\int_{\partial^+ B_r}y^{a}|\partial_\tau v|^2\,d\mathcal{H}^1&\leq 12 \int_{B_1^+}y^{a}|\nabla v|^2d{\bf x}\,.
\end{align*}
Setting $v_r({\bf x}):=v(r{\bf x})$ for ${\bf x}\in\partial^+B_1$, we deduce by scaling that $v_r$ satisfies  \eqref{smallcondext} for $\varepsilon_1$ small enough. 
Denote by $w_r$ the extension of $v_r$ provided by Lemma \ref{lemext}, and set $w({\bf x}):=w_r({\bf x}/r)$ for ${\bf x}\in B_r^+$. Scaling back, we discover that 
\begin{multline*}
\int_{B_r^+}y^{a}|\nabla w|^2\,d{\bf x}\\\leq C\left (\int_{\partial^+ B_r}y^{a}|\partial_\tau v|^2\,d\mathcal{H}^1\right )^{1/2}\left (\int_{\partial^+ B_r}y^{-a}|v-\xi|^2\,d\mathcal{H}^1+
\int_{\partial^+ B_r}y^{-a}d^2_{\mathcal N}(v)\,d\mathcal{H}^1\right)^{1/2}\,.
\end{multline*}
Since $w=v$ on $\partial^+B_r$, and $w({\bf x})\in \mathcal{N}$ on $\partial^0B^+_r$, we may extend $w$ by $v$ in $B_1^+\setminus B_r^+$ to produce a competitor to minimality, that we still denote by $w$. 
Hence, we have ${\bf E}_s(v,B_1^+)\leq {\bf E}_s(w,B_1^+)$, which leads to 
\begin{align*}
\int_{B_{1/2}^+}y^{a}|\nabla v|^2\,d{\bf x}&\leq \int_{B_{r}^+}y^{a}|\nabla v|^2\,d{\bf x}\\
&\leq \int_{B_{r}^+}y^{a}|\nabla w|^2\,d{\bf x}\\
&\leq \frac{\lambda}{12} \int_{\partial^+ B_r}y^{a}|\partial_\tau v|^2\,d\mathcal{H}^1\\
&\hskip30pt +\frac{3}{\lambda}\left(\int_{\partial^+ B_r}y^{-a}|v-\xi|^2\,d\mathcal{H}^1+\int_{\partial^+ B_r}y^{-a}d^2_{\mathcal N}(v)\,d\mathcal{H}^1\right)\\
&\leq  \lambda\int_{B_1^+}y^{a}|\nabla v|^2\,d{\bf x}\\
&\hskip30pt +\frac{36}{\lambda}\left(\int_{B_1^+}y^{-a}|v-\xi|^2\,d{\bf x}+\int_{B_1^+}y^{-a}d^2_{\mathcal N}(v)\,d{\bf x}\right) 
\end{align*}
for every $\lambda\in(0,1)$. 
\end{proof}

\subsection{Small energy regularity} We shall now prove the aforementioned small energy regularity property. As usual, 
the cornerstone argument is an energy improvement under a small oscillation condition. This leads to an improved energy decay, which in turn implies H\"older continuity as in the classical Morrey's lemma.  

\begin{theorem}[\bf Energy improvement]\label{thmdecay}
There exist constants $r_0=r_0(s,\mathcal{N})\in (0,1/2)$ and $\varepsilon_2=\varepsilon_2(s,\mathcal{N})>0$ such that the following holds. 
If $v\in H^1(B_1^+;\R^d,y^ad{\bf x})$ is a minimizing weighted harmonic map in $B_1^+$ satisfying  ${\bf E}_s(v,B_1^+)\leq \varepsilon^2_2$, 
then 
$$\frac{1}{r_0^{1-2s}}{\bf E}_s(v,B^+_{r_0})<\frac{1}{2} {\bf E}_s(v,B_1^+)\,.$$
\end{theorem} 

Let us start with the following elementary lemma inspired from \cite[Lemma 3.3]{DS}. 

\begin{lemma}\label{lemmprelepsreg}
Let $v\in H^1(B_1^+;\R^d,y^ad{\bf x})$ be such that $v({\bf x})\in \mathcal{N}$ for a.e. ${\bf x}\in \partial^0B^+_1$. Setting 
$$\bar v:= \frac{2}{\pi}\int_{B_1^+}v\,d{\bf x}\,,$$
one has
$$d_{\mathcal{N}}(\bar v)\leq C\big({\bf E}_s(v,B_1^+)\big)^{1/2} $$
for some constant $C=C(s)$. 
\end{lemma}

\begin{proof}
Since $d_{\mathcal{N}}$ is $1$-Lipschitz, we have $d_{\mathcal{N}}(\bar v)\leq |v-\bar v|+d_{\mathcal N}(v)$, and $d_{\mathcal{N}}(v)\in H^1(B_1^+,y^ad{\bf x})$ satisfies $d_{\mathcal{N}}(v)=0$ on $\partial^0B_1^+$. Applying Poincar\'e's inequalities, and H\"older's inequality,  
\begin{align*}
d_{\mathcal{N}}(\bar v) & \leq C\int_{B_1^+} |v-\bar v|\,d{\bf x}+C\int_{B_1^+} d_{\mathcal N}(v)\,d{\bf x}\\
&\leq C\int_{B_1^+} |\nabla v|\,d{\bf x}+C\int_{B_1^+} \big|\nabla\big(d_{\mathcal N}(v)\big)\big|\,d{\bf x} \\
&\leq C\big({\bf E}_s(v,B_1^+)\big)^{1/2}
\end{align*}
where we have used again the fact that $d_{\mathcal{N}}$ is $1$-Lipschitz in the last inequality.
\end{proof}

\begin{proof}[Proof of Theorem \ref{thmdecay}]
{\it Step 1.} We argue by contradiction assuming that for a given radius $r_0\in (0,1/2)$ (to be chosen), there is a sequence $\{v_n\}$ in $H^1(B_1^+;\R^d,y^ad{\bf x})$ of minimizing weighted harmonic maps in $B_1^+$ such that  
\begin{equation}\label{vanishenerg}
\varepsilon^2_n:={\bf E}_s(v_n,B_1^+)\to 0 \,,
\end{equation}
and 
\begin{equation}\label{est1425}
\frac{1}{r_0^{1-2s}}{\bf E}_s(v_n,B^+_{r_0})\geq \frac{1}{2} {\bf E}_s(v_n,B_1^+)\,.
\end{equation}
By Lemma \ref{lemmprelepsreg}, we have $d_{\mathcal{N}}(\bar v_n)\leq C\varepsilon_n\to 0$. Hence, for $n$ large enough, there is a unique $p_n\in\mathcal{N}$ such that $d_\mathcal{N}(\bar v_n)=|\bar v_n-p_n|$. Extracting a subsequence, there are $p\in\mathcal{N}$ and $q\in \R^d$ such that 
$$p_n\to p\,,\quad \bar v_n\to p\,, \;\text{and}\quad  \frac{p_n-\bar v_n}{\varepsilon_n}\to q\,.$$ 
Note that $q\in {\rm Nor}(p,\mathcal{N})$ since $p_n-\bar v_n\in {\rm Nor}(p_n,\mathcal{N})$. 
\vskip3pt

By Poincar\'e's inequality in $H^1(B_1^+;\R^d,y^ad{\bf x})$ (see  \cite{FKS}), $v_n\to p$ in $L^2(B_1^+, y^ad{\bf x})$, and therefore in $H^1(B_1^+;\R^d,y^ad{\bf x})$.   By continuity of the trace operator, we then have $v_n\to p$ in $L^2(\partial^0B_1^+)$, and thus $v_n\to p$ a.e. on $\partial^0B_1^+$, up to a further subsequence. 

Consider now the sequence
$$w_n:=\frac{1}{\varepsilon_n}(v_n-\bar v_n) $$
which satisfies 
$${\bf E}_s(w_n,B_1^+)=1\quad\text{and}\quad \int_{B_1^+}w_n\,d{\bf x}=0\,.$$
By Poincar\'e's inequality again, $\{w_n\}$ is bounded in $H^1(B_1^+;\R^d,y^ad{\bf x})$, and we can find a (not relabeled) subsequence such that $w_n\rightharpoonup w$ weakly in $H^1(B_1^+;\R^d,y^ad{\bf x})$.  
By linearity, since $v_n$ solves \eqref{yharmeq} in $B_1^+$, $w_n$ solves  \eqref{yharmeq} as well in $B_1^+$. Consequently, by weak convergence, $w$ satisfies 
\begin{equation}\label{eqintblowup}
{\rm div}(y^a\nabla w)=0\quad \text{in $B_1^+$}\,.
\end{equation}
Next, by continuity of the trace operator, we also deduce that $\{w_n\}$ is bounded in $L^2(\partial^0B_1^+)$. From Lemma \ref{H1tofracesti}, we also infer that $\{w_n\}$ is bounded in $H^s_{\rm loc}(\partial^0B_1^+)$.  By the compact embedding $H^s_{\rm loc}(\partial^0B_1^+)\hookrightarrow L^2_{\rm loc}(\partial^0B_1^+)$, we deduce that, up to a subsequence, $w_n\to w$ a.e. on $\partial^0B_1^+$ and strongly in $L^2_{\rm loc}(\partial^0B_1^+)$. For ${\bf x}\in\partial^0B_1^+$, such that $w_n({\bf x})\to w({\bf x})$ and $v_n({\bf x})\to p$, the sequence 
$$\eps^{-1}_n(v_n({\bf x})-p_n)=w_n({{\bf x}})+\eps^{-1}_n(\bar v_n-p_n)$$ is converging toward a vector in ${\rm Tan}(p,\mathcal{N})$ since $v_n({\bf x})\to p$ and $p_n\to p$. 
Therefore,
\begin{equation}\label{partDircondblowup}
w({\bf x})-q\in {\rm Tan}(p,\mathcal{N}) \text{ for a.e. ${\bf x}\in\partial^0B_1^+$}\,. 
\end{equation}
\vskip5pt

\noindent{\it Step 2.} We claim that 
\begin{equation}\label{eqblowup}
\int_{B_1^+}y^a\nabla w\cdot\nabla\zeta\,d{\bf x}=0 
\end{equation}
for every $\zeta\in C^1(\overline{B_1^+};\R^d)$ satisfying $\zeta({\bf x})\in {\rm Tan}(p,\mathcal{N})$ for every ${\bf x}\in \partial^0B_1^+$, 
and such that ${\rm spt}(\zeta)\subset B_1^+\cup\partial^0B_1^+$. 

To prove \eqref{eqblowup}, we consider the field $\Pi_b$ of $d\times d$ matrices associating to $b\in \mathcal{N}$ the orthogonal projector on ${\rm Tan}(b,\mathcal{N})$. Then we consider a (smooth) compactly supported extension of $\Pi_b$ to the whole $\R^d$. Then $\Pi_{v_n}\in H^1(B_1^+;\R^{d\times d}, y^ad{\bf x})$, and $\Pi_{v_n}\to \Pi_p$ strongly in $H^1(B_1^+;\R^{d\times d}, y^ad{\bf x})$. As a consequence, $\Pi_{v_n}\zeta\to \Pi_p\zeta$ strongly in $H^1(B_1^+;\R^{d}, y^ad{\bf x})$. Since $v_n({\bf x})\in \mathcal{N}$ for a.e. ${\bf x}\in \partial^0B_1^+$, we have  
$\Pi_{v_n({\bf x})}\zeta({\bf x})\in {\rm Tan}(u_n({\bf x}),\mathcal{N})$ for a.e. ${\bf x}\in \partial^0B_1^+$, and thus \eqref{ELfreebound} can be applied, i.e., 
$$\int_{B_1^+}y^a\nabla v_n\cdot\nabla\big(\Pi_{v_n}\zeta\big)\,d{\bf x}=0\,. $$
Therefore, 
$$\int_{B_1^+}y^a\nabla w_n\cdot\nabla\big(\Pi_{v_n}\zeta\big)\,d{\bf x}=0\,. $$
Since $\{w_n\}$ is weakly convergent and $\Pi_{v_n}\zeta$ strongly convergent, we can pass to the limit $n\to\infty$ to derive
\begin{equation}\label{eqtestproj}
\int_{B_1^+}y^a\nabla w\cdot\nabla\big(\Pi_{p}\zeta\big)\,d{\bf x}=0\,. 
\end{equation}
Since $\Pi_{p}\zeta-\zeta=0$ on $\partial^0B_1^+$, we infer from \eqref{eqintblowup} that 
\begin{equation}\label{eqdifftest}
\int_{B_1^+}y^a\nabla w\cdot\nabla\big(\Pi_{p}\zeta-\zeta\big)\,d{\bf x}=0\,.
\end{equation}
Gathering  \eqref{eqtestproj} and \eqref{eqdifftest} yields \eqref{eqblowup}.
\vskip5pt

\noindent{\it Step 3.} Set 
$$(v_n)_{2r_0}:=\int_{B^+_{2r_0}}w_n\,d{\bf x} \,,\quad (w_n)_{2r_0}:=\int_{B^+_{2r_0}}w_n\,d{\bf x}\quad\text{and}\quad  (w)_{2r_0}:=\int_{B^+_{2r_0}}w\,d{\bf x}\,. $$
Since the embedding $H^1(B_1^+,y^ad{\bf x})\hookrightarrow L^2(B_1^+,y^{-a}d{\bf x})$  is compact (see e.g. \cite{Ho}), we have Poincar\'e's inequalities telling us that 
$$\frac{1}{(2r_0)^{1+2s}}\int_{B_{2r_0}^+}y^{-a}|v_n-(v_n)_{2r_0}|^2\,d{\bf x}\leq  \frac{C}{(2r_0)^{1-2s}} {\bf E}_s(v_n,B^+_{2r_0})\leq C{\bf E}_s(v_n,B_1^+)\,,$$
and 
$$\frac{1}{(2r_0)^{1+2s}}\int_{B_{2r_0}^+}y^{-a}d^2_{\mathcal N}(v_n)\,d{\bf x}\leq \frac{C}{(2r_0)^{1-2s}} {\bf E}_s(v_n,B^+_{2r_0})\leq C{\bf E}_s(v_n,B_1^+)\,.$$
Here we have used the monotonicity formula in Lemma \ref{formmonot}, the fact that the function $d_{\mathcal N}$ is $1$-Lipschitz, and  $d_{\mathcal N}(v_n)=0$ on $\partial^0B_1^+$. 
Changing variables, one discovers that the rescaled map ${\bf x}\mapsto v_n(2r_0{\bf x})$ satisfies the small oscillation condition in Corollary~\ref{hybrid}  with $ \xi=(v_n)_{2r_0}$ for $n$ large enough, thanks to  \eqref{vanishenerg}.  
Choosing  $\lambda=1/8$ in that corollary and scaling back, we infer that 
\begin{multline}\label{eqpresfin}
\frac{1}{r_0^{1-2s}}{\bf E}_s(v_n,B^+_{r_0})\leq \frac{1}{8(2r_0)^{1-2s}}{\bf E}_s(v_n,B^+_{2r_0})\\
+\frac{C}{\lambda(2r_0)^{1+2s}}\left(\int_{B_{2r_0}^+}y^{-a}\big|v_n-(v_n)_{2r_0}\big|^2\,d{\bf x}
+\int_{B_{2r_0}^+}y^{-a}d^2_{\mathcal N}(v_n)\,d{\bf x}\right)\,.
\end{multline}
By Lemma \ref{formmonot} again, we have 
\begin{equation}\label{est1413}
\frac{1}{8(2r_0)^{1-2s}}{\bf E}_s(v_n,B^+_{2r_0})\leq \frac{\varepsilon_n^2}{8} \,.
\end{equation}
Then, 
\begin{align*}
\int_{B_{2r_0}^+}y^{-a}\big|v_n-(v_n)_{2r_0}\big|^2\,d{\bf x}&= \varepsilon_n^2\int_{B_{2r_0}^+}y^{-a}\big|w_n-(w_n)_{2r_0}\big|^2\,d{\bf x}\\
&\leq C\varepsilon_n^2\bigg(\int_{B_{2r_0}^+}y^{-a}\big|w-(w)_{2r_0}\big|^2\,d{\bf x}+\\
\int_{B_{2r_0}^+}y^{-a}&\big|w-w_n\big|^2\,d{\bf x}+\int_{B_{2r_0}^+}y^{-a}\big|(w)_{2r_0}-(w_n)_{2r_0}\big|^2\,d{\bf x}\bigg)
\end{align*}
By the two compact embeddings $H^1(B_1^+,y^ad{\bf x})\hookrightarrow L^1(B_1^+)$ and $H^1(B_1^+,y^ad{\bf x})\hookrightarrow L^2(B_1^+,y^{-a}d{\bf x})$, we have 
$w_n\to w$ strongly in $L^2(B_1^+,y^{-a}d{\bf x})$ and $(w_n)_{2r_0}\to (w)_{2r_0}$. Hence, 
\begin{equation}\label{est0924}
\int_{B_{2r_0}^+}y^{-a}\big|v_n-(v_n)_{2r_0}\big|^2\,d{\bf x}= \varepsilon_n^2\int_{B_{2r_0}^+}y^{-a}\big|w-(w)_{2r_0}\big|^2\,d{\bf x}+o(\varepsilon_n^2)\,.
\end{equation}
Next we decompose the map $w$ as $w=:w^T+w^\perp$ where $w^T$ takes values in ${\rm Tan}(p,\mathcal{N})$, and   $w^{\perp}$ takes values in ${\rm Nor}(p,\mathcal{N})$.  From 
\eqref{eqintblowup} and \eqref{partDircondblowup}, we derive that 
$$\begin{cases}
{\rm div}(y^a\nabla w^\perp)=0 & \text{in $B_1^+$}\,,\\
w^\perp= q & \text{on $\partial^0B_1^+$}\,.
\end{cases}$$
From the boundary condition, we can reflect oddly the map $(w^\perp-q)$ to the whole ball $B_1$, so that  the resulting $w^\perp$ belongs to $H^1(B_1,y^ad{\bf x})$ and satisfies 
$${\rm div}(|y|^a\nabla w^\perp)=0\quad\text{in $B_1$}\,.$$ 
By the regularity result in \cite{FKS}, $w^\perp$ is $\alpha$-H\"older continuous in $\overline B_{1/2}$ for some H\"older exponent $\alpha=\alpha(s)\in(0,1)$.  Consequently, 
\begin{equation}\label{est0925}
\int_{B_{2r_0}^+}y^{-a}\big|w^\perp-(w^\perp)_{2r_0}\big|^2\,d{\bf x}\leq Cr_0^{1+2s+2\alpha}\,.
\end{equation}
Next we deduce from \eqref{eqblowup} that 
$$\int_{B_1^+}y^a\nabla w^T\cdot\nabla \zeta\,d{\bf x}=0 $$
for every $\zeta\in C^1(\overline B_1^+;{\rm Tan}(p,\mathcal{N}))$ such that ${\rm spt}(\zeta)\subset B_1^+\cup\partial^0B_1^+$. If we reflect evenly $w^T$ to the whole ball $B_1$, 
then $w^T$ belongs to $H^1(B_1,|y|^ad{\bf x})$ and satisfies 
$${\rm div}(|y|^a\nabla w^T)=0\quad\text{in $B_1$}\,.$$ 
Once again, \cite{FKS} tells us that $w^T$ is $\alpha$-H\"older continuous in $\overline B_{1/2}$, and thus 
\begin{equation}\label{est0926}
\int_{B_{2r_0}^+}y^{-a}\big|w^T-(w^T)_{2r_0}\big|^2\,d{\bf x}\leq Cr_0^{1+2s+2\alpha}\,.
\end{equation}
In view of \eqref{est0924}, \eqref{est0925} and \eqref{est0926},  we have proved that 
\begin{equation}\label{est1418}
\int_{B_{2r_0}^+}y^{-a}\big|v_n-(v_n)_{2r_0}\big|^2\,d{\bf x}\leq C\eps_n^2r_0^{1+2s+2\alpha}+o(\varepsilon_n^2)\,.
\end{equation}

Finally, to estimate the last term in the right hand side of \eqref{eqpresfin}, we proceed as follows. First notice that 
$d_{\mathcal{N}}(v_n)\leq \eps_n|w_n| +|\bar v_n-p_n|$, so that $\eps_n^{-1}d_{\mathcal{N}}(v_n)$ is bounded in $L^2(B_1^+,y^ad{\bf x})$. Since $d_{\mathcal{N}}$ is $1$-Lipschitz, 
we have $|\nabla d_{\mathcal{N}}(v_n)|\leq \eps_n|\nabla w_n|$, and thus $\{\eps_n^{-1}d_{\mathcal{N}}(v_n)\}$ is bounded in $H^1(B^+_1,y^ad{\bf x})$. Since the embedding $H^1(B_1^+,y^ad{\bf x})\hookrightarrow L^2(B_1^+,y^{-a}d{\bf x})$ is compact,  we can assume that $\eps_n^{-1}d_{\mathcal{N}}(v_n)\to d$ in $L^2(B_1^+,y^{-a}d{\bf x})$ for some function $d\in H^1(B_1^+,y^ad{\bf x})$. Up to a further subsequence, we also have $v_n({\bf x})\to a$, $w_n({\bf x})\to w({\bf x})$, and $\eps_n^1d_{\mathcal{N}}(v_n({\bf x}))\to d({\bf x})$ for a.e. ${\bf x}\in B_1^+$. 

Given ${\bf x}\in B_1^+$ such that these convergences hold at ${\bf x}$, we have 
$$\eps_n^{-1}(v_n({\bf x})-p_n)=w_n({\bf x})+\eps_n^{-1}(\bar v_n-p_n)\to w({\bf x})-q\,.$$ 
On the other hand, for $n$ large enough, $v_n({\bf x})$ has a unique nearest point ${\bf v}_n\in\mathcal{N}$, and ${\bf v}_n\to p$.   Since $|v_n({\bf x})-p_n|\geq d_{\mathcal{N}}(v_n({\bf x}))=|v_n({\bf x})-{\bf v}_n|$ and $v_n({\bf x})-{\bf v}_n\in {\rm Nor}({\bf v}_n,\mathcal{N})$, 
$\eps_n^{-1}(v_n({\bf x})-{\bf v}_n)\to {\bf n}$ for some ${\bf n}\in {\rm Nor}(p,\mathcal{N})$, taking a subsequence if necessary. In turn, it implies that ${\eps_n^{-1}}({\bf v}_n-p_n)$ is converging toward a vector ${\bf t}\in {\rm Tan}(p,\mathcal{N})$. Consequently, ${\bf t}+{\bf n}=w({\bf x})-q$, so that ${\bf n}=w^\perp({\bf x})-q$, and thus $d({\bf x})=|w^{\perp}(x)-q|$. 

We have thus shown that $\eps_n^{-1}d_{\mathcal{N}}(v_n)\to |w^\perp -q|$ a.e. in $B^+_{1}$, and therefore in $L^2(B_1^+,y^{-a}d{\bf x})$. Hence, 
$$\int_{B_{2r_0}^+} y^{-a}d^2_{\mathcal{N}}(v_n)\,d{\bf x}=\eps_n^2\int_{B_{2r_0}^+} y^{-a}|w^\perp-q|^2\,d{\bf x} +o(\eps_n^2)\,.$$
Since $w^\perp$ is $\alpha$-H\"older continuous in $\overline B^+_{1/2}$ and $w^\perp-q=0$ on $\partial^0B_1^+$, we conclude that 
\begin{equation}\label{est1419}
\int_{B_{2r_0}^+} y^{-a}d^2_{\mathcal{N}}(v_n)\,d{\bf x}\leq C\eps_n^2r_0^{1+2s+2\alpha} +o(\eps_n^2)\,.
\end{equation}

Gathering \eqref{eqpresfin}, \eqref{est1413}, \eqref{est1418}, and \eqref{est1419} yields 
$$\frac{1}{r_0^{1-2s}}{\bf E}_s(v_n,B^+_{r_0})\leq \frac{\varepsilon_n^2}{8}+C\eps_n^2r_0^{2\alpha}+o(\varepsilon_n^2)\,.$$
Choosing $r_0$ small enough (in such a way that $Cr_0^{2\alpha}\leq 1/8$), we conclude that 
$$ \frac{1}{r_0^{1-2s}}{\bf E}_s(v_n,B^+_{r_0})< \frac{\eps_n^2}{2}$$
for $n$ large enough, contradicting \eqref{est1425}. 
\end{proof}

Arguing exactly as \cite[Theorem 2.5]{HL}, we infer from Theorem \ref{thmdecay} the following decay estimate. 

\begin{corollary}[\bf Energy decay]\label{cordecay}
If $v\in H^1(B_{2R}^+;\R^d,y^ad{\bf x})$ is a minimizing weighted harmonic map in $B_{2R}^+$ satisfying  ${\bf E}_s(v,B_{2R}^+)\leq \varepsilon^2_2 R^{1-2s}$, 
then 
$$\frac{1}{r^{1-2s}}{\bf E}_s\big(v,B_r^+({\bf x})\big)\leq CR^{-\beta} r^\beta\quad\text{for all ${\bf x}\in \partial^0B_R^+$ and $0<r\leq R$}\,,$$
for some exponent $\beta\in(0,1)$ depending only on $s$ and $\mathcal{N}$.  
\end{corollary}

In turn, this last corollary implies H\"older continuity at the boundary as in  Morrey's lemma. 

\begin{corollary}\label{cordecay2}
In addition to Corollary \ref{cordecay}, $v$ is H\"older continuous on $\partial^0B_R^+$ with H\"older exponent $\beta/2$. 
\end{corollary}

\begin{proof}
Combining Corollary \ref{cordecay} with Lemma \ref{H1tofracesti}, we first infer that 
$$\frac{1}{r^{1-2s}}\iint_{\omega_r(x_0)\times\omega_r(x_0)}\frac{|v(x)-v(y)|^2)}{|x-y|^{1+2s}}\,dxdy\leq CR^{-\beta} r^\beta $$
for every ${\bf x}_0=(x_0,0)\in \partial^0B_R^+$ and $0<r\leq R$, where we have set $\omega_r(x_0):=\partial^0B_r^+({\bf x}_0)$. Setting 
$$(v)_{x_0,r}:=\frac{1}{2r}\int_{\omega_r(x_0)} v\,dx\,, $$
we deduce from Poincar\'e's inequality in $H^s(\omega_r(x_0))$ that 
\begin{multline*}
\frac{1}{r^2}\int_{\omega_r(x_0)} |v-(v)_{x_0,r}|^2\,dx
\leq\frac{C}{r^{1-2s}}\iint_{\omega_r(x_0)\times\omega_r(x_0)}\frac{|v(x)-v(y)|^2}{|x-y|^{1+2s}}\,dxdy\\
\leq CR^{-\beta}r^\beta\,,
\end{multline*}
for all $x_0\in \omega_R(0)$ and $0< r \leq R$. The conclusion then follows from Campanato's criterion (see e.g. \cite[Chapter 6.1]{Mag}). 
\end{proof}

\section{Compactness of minimizing $s$-harmonic maps} \label{compactnes}  

This section is devoted to  compactness of minimizing $s$-harmonic maps. As it will be clear in a few lines, the proof is here much simpler compare to classical harmonic maps, as minimality can be directly tested (as if the exterior condition were fixed). Consequences concerning the extensions and densities are then easy exercices.

\begin{theorem}\label{thmcomp}
Let $\omega\subset\R$ be a bounded open interval, and $\{u_n\}_{n\in\mathbb{N}}\subset \widehat H^s(\omega;\mathcal{N})$ a sequence of minimizing $s$-harmonic maps in $\omega$. 
Assume that $\sup_n \mathcal{E}_s(u_n,\omega)<\infty$, and $u_n\to u$ in $L^2_{\rm loc}(\R)$. Then $u\in \widehat H^s(\omega;\mathcal{N})$ is a minimizing $s$-harmonic map in $\omega$, $u_n\to u$ strongly in $H^s_{\rm loc}(\omega)$, and $\mathcal{E}_s(u_n,\omega^\prime)\to \mathcal{E}_s(u,\omega^\prime)$ for every open interval such that $\overline\omega^\prime\subset\omega$.  
\end{theorem}

\begin{proof}
First we select a subsequence $u_k:=u_{n_k}$ such that $u_k\to u$ a.e. on $\R$, and 
$$\lim_{k\to\infty}\mathcal{E}_s(u_k,\omega)=\liminf_{n\to\infty}\mathcal{E}_s(u_n,\omega)<\infty\,. $$
Since each $u_k$ takes values into $\mathcal{N}$, we infer from the pointwise convergence that $u(x)\in\mathcal{N}$ for a.e. $x\in\R$. Then, by Fatou's lemma, we have 
$$ \mathcal{E}_s(u,\omega)\leq \lim_{k\to\infty}\mathcal{E}_s(u_k,\omega)\,,$$
so that $u\in \widehat H^s(\omega;\mathcal{N})$. 

Let us now consider $\widetilde u \in \widehat H^s(\omega;\mathcal{N})$ such that ${\rm spt}(u-\widetilde u)\subset \omega$. We select an open interval $\omega^\prime$ such that  
${\rm spt}(u-\widetilde u)\subset \omega^\prime$ and $\overline\omega^\prime\subset\omega$. Define
$$\widetilde u_k(x):=\begin{cases}
\widetilde u(x) & \text{if $x\in \omega^\prime$}\,,\\
u_k & \text{if $x\in \R\setminus\omega^\prime$}\,.
\end{cases} $$
It is elementary to check that $\widetilde u_k\in \widehat H^s(\omega;\mathcal{N})$, and of course ${\rm spt}(\widetilde u_k-u_k)\subset \omega$.  By minimality of $u_k$, we have 
$\mathcal{E}_s(u_k,\omega)\leq \mathcal{E}_s(\widetilde u_k,\omega)$ which leads to 
\begin{multline}\label{testminseq}
\mathcal{E}_s(u_k,\omega^\prime)\leq \mathcal{E}_s(\widetilde u_k,\omega^\prime)
=\frac{\gamma_s}{2}\iint_{\omega^\prime\times\omega^\prime}\frac{|\widetilde u(x)-\widetilde u(y)|^2}{|x-y|^{1+2s}}\,dxdy\\
+\gamma_s \iint_{\omega^\prime\times(\R\setminus\omega^\prime)}\frac{|\widetilde u(x)-u_k(y)|^2}{|x-y|^{1+2s}}\,dxdy\,.
\end{multline}
Since $\widetilde u$ and $u_k$ are taking values in $\mathcal{N}$, we have 
$$\frac{|\widetilde u(x)-u_k(y)|^2}{|x-y|^{1+2s}} \leq  \frac{C}{|x-y|^{1+2s}} \in L^1\big(\omega^\prime\times(\R\setminus\omega^\prime)\big)\,.$$
Hence $\mathcal{E}_s(\widetilde u_k,\omega^\prime)\to \mathcal{E}_s(\widetilde u,\omega^\prime)$ by dominated convergence and the fact that $\widetilde u=u$ a.e. in $\R\setminus\omega^\prime$. On the other hand, $\liminf_k\mathcal{E}_s(u_k,\omega^\prime)\geq \mathcal{E}_s(u,\omega^\prime)$, still by Fatou's lemma. Letting $k\to \infty$ in \eqref{testminseq}, we can now conclude that 
$\mathcal{E}_s(u,\omega^\prime)\leq  \mathcal{E}_s(\widetilde u,\omega^\prime)$. Once again, since $\widetilde u=u$ a.e. in $\R\setminus\omega^\prime$, this yields $\mathcal{E}_s(u,\omega)\leq  \mathcal{E}_s(\widetilde u,\omega)$. We have thus proved that $u$ is a minimizing $s$-harmonic map in $\omega$. 

In addition, the argument above applied to $\widetilde u=u$ shows that $\mathcal{E}_s(u_k,\omega^\prime)\to \mathcal{E}_s(u,\omega^\prime)$. In turn, 
$$  \iint_{\omega^\prime\times(\R\setminus\omega^\prime)}\frac{|u_k(x)-u_k(y)|^2}{|x-y|^{1+2s}}\,dxdy\to  \iint_{\omega^\prime\times(\R\setminus\omega^\prime)}\frac{|u(x)-u(y)|^2}{|x-y|^{1+2s}}\,dxdy\,,$$
again by dominated convergence. Hence, 
\begin{equation}\label{cvhsnorm}
\iint_{\omega^\prime\times\omega^\prime}\frac{|u_k(x)-u_k(y)|^2}{|x-y|^{1+2s}}\,dxdy\to \iint_{\omega^\prime\times\omega^\prime}\frac{|u(x)-u(y)|^2}{|x-y|^{1+2s}}\,dxdy \,.
\end{equation}
Since $\{u_k\}$ is bounded in $H^s(\omega^\prime)$ and $u_k\to u$ pointwise., we have  $u_k\rightharpoonup u$ weakly in $H^s(\omega^\prime)$. 
Then   \eqref{cvhsnorm} implies that $u_k\to u$ strongly in $H^s(\omega^\prime)$. 
\end{proof}

\begin{theorem}\label{compext}
In addition to Theorem \ref{thmcomp}, $u_n^{\rm e}\to u^{\rm e}$ strongly in $H^1(\Omega;\R^d, y^ad{\bf x})$ for every bounded admissible open set $\Omega\subset\R^2_+$ such that $\overline{\partial^0\Omega}\subset\omega$. 
\end{theorem}

\begin{proof}
From Theorem \ref{thmcomp} and \cite[Lemma 2.10]{MSW}, we start deducing that $u_n^{\rm e}\to u^{\rm e}$ strongly in $L^2_{\rm loc}(\overline\R^2_+;\R^d, y^ad{\bf x})$. Since $u^{\rm e}$ solves 
\eqref{eqextint}, we infer from standard elliptic theory that $u_n^{\rm e}\to u^{\rm e}$ strongly in $H^1_{\rm loc}(\R^2_+;\R^d, y^ad{\bf x})$. It remains to prove that strong convergence holds up to $\partial^0\Omega$ (locally). To this purpose, let us fix an arbitrary half ball $B^+_r({\bf x_0})$ such that ${\bf x}_0\in \partial^0\Omega$ and  $\partial^0B^+_{3r}({\bf x_0})\subset \omega$. By \cite[Lemma 2.10]{MSW}, we have 
$${\bf E}_s\big(u_n^{\rm e}-u^{\rm e},B^+_r({\bf x_0})\big) \leq C\left(\mathcal{E}_s(u_n-u,\partial^0B_{2r}^+({\bf x}_0))+\|u_n-u\|_{L^2(\partial^0B_{2r}^+({\bf x}_0))}\right)\to 0\,,$$
again by Theorem  \ref{thmcomp}. 
\end{proof}

\begin{corollary}\label{coroldenssup}
In addition to Theorem \ref{thmcomp}, if $\{x_n\}\subset \omega$ is a sequence converging to $x\in\omega$, then 
$$\limsup_{n\to\infty}\mathbf{\Theta}_{u_n^{\rm e}}(x_n)\leq  \mathbf{\Theta}_{u^{\rm e}}(x)\,.$$
\end{corollary}

\begin{proof}
Without loss of generality we may assume that $x=0$. For $r>0$ small enough we have $\overline{\partial^0 B^+_{2r}}\subset \omega$. Setting $r_n:=|x_n|$, we have $r_n<r$ for $n$ large enough. Then, we infer from Corollary \ref{corodens} that
$$\mathbf{\Theta}_{u_n^{\rm e}}(x_n)\leq  \mathbf{\Theta}_{u_n^{\rm e}}(x_n,r)\leq \frac{1}{r^{1-2s}}{\bf E}_s(u^{\rm e}_n,B^+_{r+r_n})\,.$$
By Theorem \ref{compext}, we have $u_n^{\rm e}\to u^{\rm e}$ strongly in $H^1(B^+_{2r};\R^d, y^a\,d{\bf x})$, and thus 
$$\limsup_{n\to\infty}\mathbf{\Theta}_{u_n^{\rm e}}(x_n)\leq \lim_{n\to\infty}\frac{1}{r^{1-2s}}{\bf E}_s(u^{\rm e}_n,B^+_{r+r_n})= \mathbf{\Theta}_{u^{\rm e}}(0,r) \,.$$
Letting now $r\downarrow0$ provides the desired conclusion. 
\end{proof}

\section{Proof of Theorems \ref{main1} \& \ref{main2} }\label{prfthms}  

This section is devoted to the proof of Theorem \ref{main1} and \ref{main2}. We consider for the entire section a bounded open interval $\omega\subset\R$, and 
$u\in \widehat H^s(\omega;\mathcal{N})$ a minimizing $s$-harmonic map in $\omega$. Both proofs rely on the analysis of {\sl tangent maps} of $u$ at a given point of $\omega$. 
To define them, we  fix a point $x_0\in \omega$, and for $\rho>0$ we consider the rescaled map
$$u_{x_0,\rho}(x):=u(x_0+\rho x)\,.$$  
Tangent maps of $u$ at $x_0$ are all possible weak limits of $u_{x_0,\rho}$ as $\rho\downarrow 0$, and this is is the purpose of the following proposition. 

\begin{proposition}[\bf Tangent maps]\label{tangmap}
Let $\rho_n\to 0$ be an arbitrary sequence. There is a (not relabeled) subsequence such that $u_{x_0,\rho_n}\to u_0$ strongly in $H^s_{\rm loc}(\R)$, where $u_0$ is a minimizing $s$-harmonic map in every bounded open interval of the form 
\begin{equation}\label{formtangmap}
u_0(x):=\begin{cases} 
a & \text{if $x>0$}\,,\\
b & \text{if $x<0$}\,,
\end{cases}
\end{equation}
for some $a,b\in\mathcal{N}$. In addition, $\mathbf{\Theta}_{u_0^{\rm e}}(0,r)=\mathbf{\Theta}_{u_0^{\rm e}}(0)=\mathbf{\Theta}_{u^{\rm e}}(x_0)$ for every $r>0$. 
\end{proposition}

\begin{proof}
Assume  without loss of generality that $x_0=0$ and $[-1,1]\subset\omega$.  For an integer $k\geq 1$, write $\omega_k:=(-k,k)$. For $n$ large enough,  $2\rho_nk\leq 1$ and $u_n:=u_{0,\rho_n}\in \widehat H^s(\omega_k;\mathcal{N})$.   Moreover, 
$$\mathcal{E}_s(u_n,\omega_k)=\frac{1}{\rho_n^{1-2s}} \mathcal{E}_s(u,\rho_n\omega_k)\,.$$ 
Next we infer from Lemma~\ref{H1tofracesti} and Lemma \ref{formmonot} that 
$$\frac{1}{\rho_n^{1-2s}}\iint_{(\rho_n\omega_k)\times(\rho_n\omega_k)}\frac{|u(x)-u(y)|^2}{|x-y|^{1+2s}}\,dxdy\leq C\mathbf{\Theta}_{u^{\rm e}}(0,2\rho_nk) 
\leq  C\mathbf{\Theta}_{u^{\rm e}}(0,1)\,.$$
On the other hand, 
\begin{multline*}
\frac{1}{\rho_n^{1-2s}}\iint_{(\rho_n\omega_k)\times(\R\setminus\rho_n\omega_k)}\frac{|u(x)-u(y)|^2}{|x-y|^{1+2s}}\,dxdy\\
\leq \frac{C}{\rho_n^{1-2s}}\iint_{(\rho_n\omega_k)\times(\R\setminus\rho_n\omega_k)}\frac{1}{|x-y|^{1+2s}}\,dxdy=Ck^{1-2s}\,. 
\end{multline*}
Therefore $\mathcal{E}_s(u_n,\omega_k)\leq C_k$ for a constant $C_k$ depending only on $s$ and $k$.  In particular, $\{u_n\}$ is bounded in $H^s(\omega_k)$ for each integer $k\geq 1$. Hence, we can find a (not relabeled) subsequence such that $u_k\rightharpoonup u_0$ weakly in $H^s_{\rm loc}(\R)$. From the compact embedding $H^s(\omega_k)\hookrightarrow L^2(\omega_k)$, we also deduce that 
$u_n\to u_0$ in $L^2_{\rm loc}(\R)$.  Applying Theorem \ref{thmcomp} in each $\omega_k$, we derive that $u_0$ is a minimizing $s$-harmonic map in every bounded open interval. 
Next, Theorem \ref{compext}  implies that 
$$\mathbf{\Theta}_{u_0^{\rm e}}(0,r)=\lim_{n\to\infty}\mathbf{\Theta}_{u_n^{\rm e}}(0,r)=\lim_{n\to\infty}\mathbf{\Theta}_{u^{\rm e}}(0,\rho_nr)=\mathbf{\Theta}_{u^{\rm e}}(0)\quad\forall r>0\,. $$
Here, we have also used that $u^{\rm e}_n({\bf x})=u^{\rm e}(\rho_n{\bf x})$. In view of Corollary~\ref{corodens}, we thus have 
$$\int_{B^+_r}y^a\frac{|{\bf x}\cdot\nabla u_0^{\rm e}|^2}{|{\bf x}|^{3-2s}}\,d{\bf x}= \mathbf{\Theta}_{u_0^{\rm e}}(0,r)-\mathbf{\Theta}_{u_0^{\rm e}}(0)=0\quad\forall r>0\,. $$
Therefore ${\bf x}\cdot\nabla u^{\rm e}=0$, so that $u_0^{\rm e}$ is positively $0$-homogeneous, i.e., $u_0^{\rm e}(\lambda {\bf x})=u_0^{\rm e}({\bf x})$ for every ${\bf x}\in \R^2_+$ and $\lambda>0$. In particular, $u_0$ is positively $0$-homogeneous, and \eqref{formtangmap} follows. 
\end{proof}

\begin{remark}\label{remdenscont}
If $u$ is continuous at $x_0$, the limit $u_0$ obtained in Proposition \ref{tangmap} is obviously the constant map equal to $u(x_0)$. As a consequence, if $u$ is continuous at~$x_0$, then $\mathbf{\Theta}_{u^{\rm e}}\big((x_0,0)\big)=0$.   
\end{remark}

\begin{proof}[Proof of Theorem \ref{main1}] Let us consider the set
\begin{equation}\label{defsingset}
S:=\big\{x\in\omega: \mathbf{\Theta}_{u^{\rm e}}\big((x,0)\big)\geq 2^{2s-1}\varepsilon^2_2\big\}\,,
\end{equation}
where $\varepsilon_2>0$ is the constant given by Theorem~\ref{thmdecay}. Since $\mathbf{\Theta}_{u^{\rm e}}$ is upper semicontinuous, $S$ is a relatively closed subset of $\omega$.  Moreover, Corollaries \ref{cordecay} \& \ref{cordecay2} together with Corollary~\ref{corodens} implies that $u$ is locally H\"older continuous in $\omega\setminus S$. To prove Theorem\ref{main1}, it then remains to show that $S$ has no accumulation point in $\omega$.
We argue by contradiction assuming that there is a sequence $\{x_n\}\subset S$ such that $x_n\to x\in\omega$. Without loss of generality, we may assume that $x_n>x$. Setting $\rho_n:=x_n-x$, we consider the sequence $u_n:=u_{x_0,\rho_n}$, and then apply Proposition \ref{tangmap} to find a (not relabeled) subsequence and a minimizing $s$-harmonic map $u_0$ of the form   \eqref{formtangmap}  such that $u_n\to u_0$.   In view of Corollary \ref{coroldenssup} we have 
$$\mathbf{\Theta}_{u_0^{\rm e}}\big((1,0)\big)\geq \limsup_{n\to\infty} \mathbf{\Theta}_{u_n^{\rm e}}\big((1,0)\big)= \limsup_{n\to\infty} \mathbf{\Theta}_{u^{\rm e}}\big((x_n,0)\big)\geq \varepsilon_2\,. $$
On the other hand, by the explicit form \eqref{formtangmap}, the map $u_0$ is continuous at~$1$. Hence, $\mathbf{\Theta}_{u_0^{\rm e}}\big((1,0)\big)=0$ by Remark \ref{remdenscont}, contradiction.  
\end{proof}

\begin{proof}[Proof of Theorem \ref{main2}]
Recall that we assume now that $\mathcal{N}=\mathbb{S}^{d-1}$. In view of the proof of Theorem \ref{main1}, it is enough to show that the set $S$ defined in \eqref{defsingset} is empty. Assume by contradiction that $S\not=\emptyset$. We may then assume without loss of generality that $0\in S$. Let $u_0$ be a $s$-minimizing harmonic map produced by Proposition \ref{tangmap}, i.e., $u_0$ is the limit of the rescaled map $u_{0,\rho_n}$ for some sequence $\rho_n\to 0$. Then $\mathbf{\Theta}_{u_0^{\rm e}}(0)\geq \varepsilon_2>0$, so that $u_0$ is not constant. In other words, in the form \eqref{formtangmap} the two vectors $a,b\in\mathbb{S}^{d-1}$ are distinct. Upon working in the plane passing through $a$, $b$, and the origin, there is no loss of generality assuming that $d=2$, that  is $\mathcal{N}=\mathbb{S}^1$. Moreover, rotating coordinates in the image if necessary, we can assume that   
$$ a=(\alpha,\beta) \text{ and } b=(-\alpha,\beta)\,,$$
with $0<\alpha\leq 1$ and $0\leq \beta<1$ satisfying $\alpha^2+\beta^2=1$. 
Then set 
$$a^*:=(-\beta,\alpha)\text{ and } b^*:=(\beta,\alpha) \,.$$
Note that $a^*\perp a$ and $b^*\perp b$.  
We define for $t\in\R$, 
$$u_t(x):=\begin{cases} 
\displaystyle \frac{a+ta^*}{\sqrt{1+t^2}} & \text{if $0<x<1$}\,,\\[10pt]
\displaystyle \frac{b+tb^*}{\sqrt{1+t^2}}  & \text{if $-1<x<0$}\,,\\[10pt]
u_0(x) & \text{otherwise}\,.
\end{cases}$$
One can easily check that $u_t\in \widehat H^s\big((-2,2);\mathbb{S}^1\big)$, and since ${\rm spt}(u_t-u)\subset (-2,2)$, the map $u_t$ is an admissible competitor for the minimality of $u_0$ in $(-2,2)$. In other words, $\mathcal{E}_s\big(u_0,(-2,2)\big)\leq \mathcal{E}_s\big(u_t,(-2,2)\big)$, which in turn yields $\mathcal{E}_s\big(u_0,(-1,1)\big)\leq \mathcal{E}_s\big(u_t,(-1,1)\big)$ since $u_t=u_0$ outside $(-1,1)$. Therefore, 
$$\left[\frac{d}{dt}\mathcal{E}_s\big(u_t,(-1,1)\big)\right]_{t=0}=0 \quad\text{and}\quad\left[\frac{d^2}{dt^2}\mathcal{E}_s\big(u_t,(-1,1)\big)\right]_{t=0}\geq 0\,.$$
Now we expand  $\mathcal{E}_s\big(u_t,(-1,1)\big)$ as 
\begin{align*}
\mathcal{E}_s\big(u_t,(-1,1)\big) =&\left(\gamma_s\int_{0}^1\int_{-1}^0\frac{dxdy}{|x-y|^{1+2s}}\right)\frac{\big|(a-b)+t(a^*-b^*)\big|^2}{1+t^2}\\
&+\left( \gamma_s\int_0^1\int_1^{+\infty}\frac{dxdy}{|x-y|^{1+2s}} \right)\frac{\big|(1-\sqrt{1+t^2})a+ta^*\big|^2}{1+t^2}\\
&+\left(\gamma_s\int_{0}^1\int_{-\infty}^{-1}\frac{dxdy}{|x-y|^{1+2s}}\right)\frac{\big|a+ta^*-b\sqrt{1+t^2}\big|^2}{1+t^2}\ \\
&+\left(\gamma_s\int_{-1}^0\int_{-\infty}^{-1}\frac{dxdy}{|x-y|^{1+2s}}\right)\frac{\big|(1-\sqrt{1+t^2})b+tb^*\big|^2}{1+t^2} \\
&+\left(\gamma_s\int_{-1}^0\int_1^{+\infty}\frac{dxdy}{|x-y|^{1+2s}}\right)\frac{\big|b+tb^*-a\sqrt{1+t^2}\big|^2}{1+t^2} \,.
\end{align*}
It then follows that 
$$\left[\frac{d}{dt}\mathcal{E}_s\big(u_t,(-1,1)\big)\right]_{t=0}= -C\alpha \beta$$
for some constant $C=C(s)>0$. The first order condition implies $\alpha\beta=0$, and thus 
$$a=(1,0)\,,\; b=(-1,0)\,,\text{ and } a^*=b^*=(0,1)\,. $$
As a consequence, using the symmetry in the integrals above,   
\begin{align*}
\mathcal{E}_s\big(u_t,(-1,1)\big) =&\left(4\gamma_s\int_{0}^1\int_{-1}^0\frac{dxdy}{|x-y|^{1+2s}}\right)\frac{1}{1+t^2}\\
&+\left( 2\gamma_s\int_0^1\int_1^{+\infty}\frac{dxdy}{|x-y|^{1+2s}} \right)\frac{(1-\sqrt{1+t^2})^2+t^2}{1+t^2}\\
&+\left(2\gamma_s\int_{0}^1\int_{-\infty}^{-1}\frac{dxdy}{|x-y|^{1+2s}}\right)\frac{(1+\sqrt{1+t^2})^2+t^2}{1+t^2} \,.
\end{align*}
An elementary computation now yields 
\begin{align*}
\left[\frac{d^2}{dt^2}\mathcal{E}_s\big(u_t,(-1,1)\big)\right]_{t=0} &=-8\gamma_s\int_{0}^1\int_{-1}^0\frac{dxdy}{|x-y|^{1+2s}} +4\gamma_s\int_0^1\int_1^{+\infty}\frac{dxdy}{|x-y|^{1+2s}} \\
&\hskip120pt-4\gamma_s\int_{0}^1\int_{-\infty}^{-1}\frac{dxdy}{|x-y|^{1+2s}}\\
&=-4\gamma_s\int_{0}^1\int_{-1}^0\frac{dxdy}{|x-y|^{1+2s}} +4\gamma_s\int_0^1\int_1^{+\infty}\frac{dxdy}{|x-y|^{1+2s}}\\ 
&\hskip120pt-4\gamma_s\int_{0}^1\int_{-\infty}^{0}\frac{dxdy}{|x-y|^{1+2s}}\\
&= -4\gamma_s\int_{0}^1\int_{-1}^0\frac{dxdy}{|x-y|^{1+2s}}\\
& <0\,,
\end{align*}
contradicting the second order condition for minimality. 
\end{proof}

\vskip10pt
\subsubsection*{Aknowledgements}
V.M. is supported by the {\sl Agence Nationale de la Recherche} through the projects  ANR-12-BS01-0014-01 (Geometrya) and  ANR-14-CE25-0009-01 (MAToS). 

\vskip10pt


\end{document}